\documentclass[onefignum,onetabnum]{siamart171218}

\usepackage[T1]{fontenc}
\usepackage{lipsum}
\usepackage{amsfonts}
\usepackage{graphicx}
\usepackage{epstopdf}
\usepackage{algorithmic}
\usepackage{amsmath}
\usepackage{colortbl}
\usepackage{booktabs}
\usepackage{enumitem}
\usepackage{multirow}

\makeatletter
\def\refstepcounter@optarg[#1]#2{%
  \cref@old@refstepcounter{#2}%
  \cref@constructprefix{#2}{\cref@result}%
  \@ifundefined{cref@#1@alias}%
    {\def\@tempa{#1}}%
    {\def\@tempa{\csname cref@#1@alias\endcsname}}%
  \protected@edef\cref@currentlabel{%
    [\@tempa][\arabic{#2}][\cref@result]%
    \csname p@#2\endcsname\csname the#2\endcsname}%
}
\makeatother

\ifpdf
  \DeclareGraphicsExtensions{.eps,.pdf,.png,.jpg}
\else
  \DeclareGraphicsExtensions{.eps}
\fi

\newsiamremark{remark}{Remark}
\newsiamremark{hypothesis}{Hypothesis}
\newsiamthm{claim}{Claim}

\crefname{hypothesis}{Hypothesis}{Hypotheses}

\usepackage{amsopn}
\DeclareMathOperator{\diag}{diag}

\makeatletter
\newcommand*{\addFileDependency}[1]{
  \typeout{(#1)}
  \@addtofilelist{#1}
  \IfFileExists{#1}{}{\typeout{No file #1.}}
}
\makeatother

\newcommand*{\myexternaldocument}[1]{%
    \externaldocument{#1}%
    \addFileDependency{#1.tex}%
    \addFileDependency{#1.aux}%
}

\headers{Adaptive Polynomial Filtering }{Hermitian Interior Eigenvalue problems}

\title{Adaptive Polynomial Filtering for Hermitian Interior Eigenproblems: Convergence Analysis
}

\author{
Xiaofei Xu
\and
Yuhui Ni
\and
Shengguo Li
\and
Juan Zhang
}

\ifpdf
\hypersetup{
  pdftitle={      },
  pdfauthor={}
}
\fi

\myexternaldocument{ex_supplement}

\begin{document}

\maketitle


\begin{abstract}
Interior eigenvalue problems for large-scale sparse Hermitian matrices are fundamental in computational science. 
We propose an adaptive polynomial filtering strategy based on Chebyshev expansion of a step function, integrated into a filtered subspace iteration framework. 
We establish pointwise convergence bounds in both undamped and damped settings and incorporate an enhanced spurious eigenvalue detection technique to improve efficiency and robustness. 
At the implementation level, we employ MaSpMM to accelerate the polynomial filtering step. 
Numerical results demonstrate the efficiency and robustness of the proposed method compared with classical approaches.
\end{abstract}

\begin{keywords}
    Subspace iteration, Adaptive polynomial filtering, SpMM, Spurious eigenvalues, Interior eigenvalue problems 
\end{keywords}


\section{Introduction}


Interior eigenvalue problems for large-scale sparse Hermitian matrices arise in many areas of computational science and engineering. A prominent driving application is Kohn--Sham density-functional theory (DFT) \cite{hohenberg1964inhomogeneous,kohn1965self}, in which each step of the self-consistent field iteration requires computing the occupied eigenstates of a discretized Hamiltonian in a neighborhood of the Fermi level. Because these eigenvalues lie in the interior of the spectrum and the Hamiltonian dimension can reach millions, efficient interior eigensolvers are indispensable for large-scale DFT calculations \cite{zhou2006self,liou2020scalable}. Interior eigenvalue problems also appear in quantum many-body and condensed-matter simulations \cite{van2022enhancing}, photonic band-structure and waveguide computations \cite{lyu2025contour}, and 
planetary normal-mode computations \cite{shi2018computing}.
Furthermore, for problems of interest to the full spectrum, the spectrum can be partitioned into slices, with the eigenpairs in each slice computed in parallel.

Given a Hermitian matrix $A \in \mathbb{C}^{n\times n}$ and an interval of interest $[a, b] \subset [\lambda_{\text{min}}, \lambda_{\text{max}}]$, 
let $\{\lambda_j\}_{j=1}^n$ denote the eigenvalues of $A$, with corresponding eigenvectors $v_j$ satisfying $\langle v_i, v_j \rangle = \delta_{ij}$. We consider the \emph{interior eigenvalue problem}
\begin{equation}\label{ori_problem}
    A v_j = \lambda_j v_j,
\end{equation} 
where $\lambda_j \in [a, b] $.
If the target eigenvalues are the first \(e\) (or the last \(e\)) eigenvalues of \(A\), the problem \eqref{ori_problem} reduces to an \emph{extreme eigenvalue problem}, which can be regarded as a special case of the interior eigenvalue problem.

For small- and medium-scale dense matrices, interior eigenvalue problems have been studied extensively. Classical approaches include bisection with inverse iteration \cite{demmel1997applied} and the MRRR algorithm \cite{dhillon2006design}. 
These methods rely on matrix factorizations that incur substantial memory costs and become impractical for large-scale problems. 
Structured variants based on hierarchical low-rank representations, such as HSS-based $LDL^T$ factorizations \cite{xi2014fast}, 
have been proposed to accelerate such computations, but they still require explicit matrix representations. 
For large-scale sparse Hermitian matrices, a variety of algorithms and solvers have been developed in recent years. 
These approaches are generally built upon projection methods combined with spectral transformations that redirect computational effort toward the desired part of the spectrum.

Among projection methods for extreme eigenvalue problems, subspace iteration \cite{bauer1957verfahren,demmel1997applied}, the Implicitly Restarted Arnoldi method (IRAM/ARPACK) \cite{lehoucq1998arpack}, and the LOBPCG method \cite{knyazev2001toward} are widely used. Chebyshev-filtered subspace iteration (ChFSI) \cite{zhou2006self} and its high-performance implementation ChASE \cite{winkelmann2019chase} have proved particularly effective for sequences of Hermitian eigenvalue problems arising, e.g., in self-consistent field iterations. These projection methods can be extended to interior eigenvalue problems through spectral transformations. A classical example is the shift-and-invert transformation, which leads to shift-and-invert Lanczos/Arnoldi methods and shift-and-invert Jacobi--Davidson variants \cite{sleijpen2000jacobi}.

Filtering provides another important class of spectral transformations that maps unwanted spectral components toward zero. 
Among rational filtering approaches, the Sakurai--Sugiura (SS) method \cite{sakurai2003projection} and 
its variants \cite{ikegami2010contour,ikegami2010filter,sakurai2007cirr, xi2016computing} form the basis of the z-Pares software, 
while the FEAST algorithm \cite{polizzi2009density} and its extensions \cite{chen2025interior,gavin2018krylov,gavin2018feast,peter2014feast} constitute the FEAST software package \cite{polizzi2020feast}. 
These contour-integral-based filters offer high accuracy and excellent parallel scalability. 
However, both the shift-and-invert approach and rational filtering require solving linear systems involving $A - \sigma I$, 
which places nontrivial demands on the choice of the shift $\sigma$.

An alternative is polynomial filtering, which, although less accurate in approximation than its rational counterpart, replaces the linear solves with sparse matrix--vector multiplications. Li et al. \cite{li2016thick} employed the Chebyshev series expansion of the Dirac delta function as a filter, integrating it into both the thick-restart Lanczos method and subspace iteration. This strategy is particularly effective for targeting narrow spectral intervals; hence, the algorithm incorporates a spectrum-slicing strategy. These algorithms are implemented in the EVSL software package \cite{li2019eigenvalues}. Similarly, Jia et al. utilized the Chebyshev-Jackson series expansion of a step function as a filter within a subspace iteration framework, leading to the CJ-FEAST algorithm \cite{jia2023feast} for SVD problems. A backward-stable variant of this algorithm was subsequently developed \cite{jia2024augmented}.

In this paper, we propose an adaptive polynomial filter integrated within a subspace iteration (AdaPolySI) framework for solving Hermitian interior eigenvalue problems. 
A preliminary version of this algorithm was presented in \cite{Yuhui2026AdaPolySI}; in this work, we further extend it with several additional contributions. 
In particular, we optimize and accelerate SpMM using Clenshaw's algorithm, provide more comprehensive theoretical analyses, and propose a more efficient criterion for excluding spurious eigenvalues.
Similar to CJ-FEAST, our filter employs a Chebyshev series expansion of a step function; 
however, a key distinction lies in our use of exponent-$m$ Lanczos damping instead of Jackson damping, where the parameter $m$ controls the damping intensity. 
The filter degree adapts during the iteration to balance computational cost and solution accuracy. 
We establish rigorous pointwise convergence bounds for the polynomial filter in both the undamped and damped settings, 
providing a theoretical foundation for the adaptive degree selection strategy. 
At the implementation level, we incorporate implicit deflation and a spurious eigenvalue detection mechanism to accelerate convergence,
and employ MaSpMM \cite{MaSpMM}, a memory-aware SpMM microkernel, to accelerate the dominant sparse matrix--dense matrix multiplication in the filtering step.

For a set of sparse matrices from the SuiteSparse~\cite{TimDavis-Matrix}, 
numerical results demonstrate that AdaPolySI achieves average speedups of $\textbf{14.5}\times$ over EVSL and $\textbf{2.14}\times$ over CJ-FEAST,
with peak speedups of up to $25.9\times$ and $2.39\times$, respectively.
We also evaluate the effectiveness of the proposed adaptive strategy and analyze the sensitivity of the involved parameters, 
as detailed in Section~\ref{sec:Numerical_experiments}.


The remainder of this paper is organized as follows. Section~\ref{Preliminaries} reviews the filtered subspace iteration framework and polynomial filtering. Section~\ref{Accuracy of the filter} develops the proposed adaptive filter and establishes its convergence properties. Section~\ref{Implementations} describes algorithmic implementation details, including deflation, spurious eigenvalue detection, and MaSpMM. Numerical experiments are reported in Section~\ref{sec:Numerical_experiments}. Finally, Section~\ref{Conclusions} summarizes our contributions and discusses future directions.

For a matrix $X \in \mathbb{C}^{m \times n}$, let $X_{i:j, k:l}$ denote the submatrix consisting of rows $i$ through $j$ and columns $k$ through $l$, and let $X^*$, $X^T$, $\text{nnz}(X)$, and $\kappa(X)$ denote its conjugate transpose, transpose, the number of non-zero entries, and condition number, respectively. 
$|\mathcal{I}|$ and \(\mathcal{I}^c\) represent the cardinality and complement of a set $\mathcal{I}$, respectively.
For a subset \( \mathcal{J} \subset \mathcal{I} \), \( \mathcal{I} \setminus \mathcal{J} \) denotes the complement of \( \mathcal{J} \) with respect to \( \mathcal{I} \).


\section{Preliminaries} \label{Preliminaries}
\subsection{Filtered Subspace Iteration framework}
\label{sec:main}


The modern Subspace Iteration (SI) framework has gone through a longer period of development.
A limitation of subspace iteration methods is their inability to target eigen-pairs in arbitrary intervals of interest, i.e. solving the interior eigenvalue problem \eqref{ori_problem}. Filtering techniques \cite{saad2011numerical} provide an effective solution. The technique constructs a function $\rho(x)$ that maps target eigenvalues of A to dominant eigenvalues of $\rho(A)$, while suppressing non-target eigenvalues to near zero. By replacing $A$ with $\rho(A)$ in subspace iterations, we can approximate the invariant subspace corresponding to the desired eigenvalues. As far as we know, existing filters fall into two primary types -- rational filters and polynomial filters. Our algorithm employs the latter approach, detailed in Section \ref{Filtering}.

\begin{algorithm}[!h]
	\caption{Filtered Subspace iteration for Hermitian eigenproblems.}
	\label{alg:SI}
	\renewcommand{\algorithmicrequire}{\textbf{Input:}}
	\renewcommand{\algorithmicensure}{\textbf{Output:}}
	\begin{algorithmic}[1]
		\REQUIRE Hermitian matrix $A$, 
                 filtering function $\rho(x)$, 
                 interval of interest $[a, b]$ and
                 the matrix $\hat V^{(i)}$.
		\ENSURE Computed eigen-pairs ($\hat\lambda_i, \hat v_i$), $\hat\lambda_i \in [a, b].$ 
        \STATE{$i = 1$}
		\WHILE{not converge}
		\STATE{$\hat X^{(i)} = \rho(A) \hat V^{(i-1)}$}
		\STATE{$Q^{(i)}$ = $\text{orth}(\hat X^{(i)})$}
        \STATE{[$\hat V^{(i)}$, $\hat\Lambda^{(i)}$] = \text{Rayleigh-Ritz}($A$,  $Q^{(i)}$)}  
        \STATE{$\hat\lambda_i = (\hat\Lambda^{(i)})_{i,i}$, $\hat v_i = (\hat V^{(i)})_{:,i}$}
        \STATE{$i = i + 1$}
		\ENDWHILE
        \STATE \textbf{function} Rayleigh-Ritz($A, Q$)
            \STATE{$B = Q^* A Q$}
            \STATE{Compute eigen-decomposition $B = \hat U \hat\Lambda \hat U^*$}
    	    \STATE{$\hat V = Q\hat U$}
            \STATE \textbf{return} $\hat V, \hat \Lambda$
        \STATE \textbf{end function}
	\end{algorithmic}
\end{algorithm}

The filtered subspace iteration (FSI) framework for Hermitian eigenvalue problem pseudo-code is shown in Algorithm \ref{alg:SI}, where it is assumed that a suitable filtering function $\rho(x)$ has already been established. The columns of the $n \times p$ matrix $\hat V^{(0)}$ form the initial orthonormal basis. Convergence of the FSI algorithm requires $p \ge e$, where $e$ is the number of eigenvalues in the interval of interest \cite{peter2014feast}. 

The main computational overhead of FSI comes from line 3. This step requires repeated multiplication of \(A\) by a block vector \(X\). As a result, the number of matrix-vector products (MVs) serves as a measure of the computational cost of FSI. For sparse \(A\), the computational complexity is \(\mathcal{O}(\mathrm{nnz}(A)) \cdot p\), and the operation can be efficiently implemented using Level~3 BLAS routines.

The QR decomposition in line 4 ensures orthogonality of the subspace basis. We denote by orth($X$) the orthonormal matrix $Q$ from the QR factorization. In line 5, Rayleigh-Ritz process projects $A$ onto a smaller matrix $B$ whose eigenvalues approximate some eigenvalues of $A$. We compute all eigen-pairs of $B$ using state-of-the-art Hermitian eigensolvers such as QR iteration or divide-and-conquer \cite{demmel1997applied}, whose high-performance implementations are available in LAPACK \cite{anderson1999lapack}.

To analyze the convergence of the FSI algorithm, we consider the following decomposition
\begin{equation*}
    A = V\Lambda V^* = V_p \Lambda_p V_p^* + V_{p,\perp} \Lambda_{p, \perp} V_{p,\perp}^*,
\end{equation*}
where
\begin{align*}
    V &= \left[V_p,\ V_{p,\perp}\right],\quad 
    \Lambda = 
    \begin{bmatrix}
    \Lambda_p & 0\\
    0 & \Lambda_{p, \perp}
    \end{bmatrix},\\
    V_p &= \left[ v_1, v_2,\cdots,v_e, \cdots, v_p\right],\quad  V_{p, \perp} = \left[ v_{p+1}, \cdots,v_n\right],\\
    \Lambda_p &= \diag(\lambda_1,\cdots,\lambda_e,\cdots,\lambda_p),\quad \Lambda_{p, \perp} = \diag(\lambda_{p+1},\cdots,\lambda_n).
\end{align*}
The filter is allowed to vary across iterations and is denoted by \(\rho^{(i)}(x)\) at iteration \(i\). Let $\gamma_j^{(i)} = \rho^{(i)}(\lambda_j)$ for $j = 1, \ldots, n$ be decreasingly ordered. Define
\[
    \Gamma_p^{(i)} = \diag\left(\gamma_1^{(i)}, \cdots, \gamma_e^{(i)},\cdots,\gamma_p^{(i)}\right), \quad 
    \Gamma_{p,\perp}^{(i)} = \diag\left(\gamma_{p+1}^{(i)}, \gamma_{p+2}^{(i)}, \cdots, \gamma_n^{(i)}\right).
\]
Let $\left(\hat\Lambda^{(i)}, \hat V^{(i)}\right)$, 
where 
\[
    \hat V^{(i)} = \left[ \hat v_1^{(i)},\cdots,\hat v_p^{(i)} \right]  \in \mathbb{C}^{n\times p}, \quad \hat\Lambda^{(i)} = \diag\left( \hat \lambda_1^{(i)},\cdots,\hat \lambda_p^{(i)} \right)   \in \mathbb{R}^{p\times p},
\]
be the approximate eigen-pairs obtained at the $i$-th iteration. 
The convergence properties of the filtered subspace iteration are reviewed in the following lemma.
\begin{lemma}\label{lem:convergence}{\rm \cite{jia2023feast,saad2011numerical}}
Assume $|\gamma_e^{(i)}| > |\gamma_{p+1}^{(i)}|$ and that $V_p^* \hat{V}^{(0)}$ is nonsingular.
Define  
\[
E^{(0)} = \left(V_{p,\perp}^* \hat{V}^{(0)}\right)\left(V_p^* \hat{V}^{(0)}\right)^{-1}, \quad E^{(i)} = \Gamma_{p,\perp}^{(i)} E^{(i-1)} \left(\Gamma_p^{(i)}\right)^{-1}.
\]
Then, the following statements hold:
\begin{enumerate}
    \item Let \( \epsilon^{(i)} = \mathrm{dist}\left\{\mathrm{span}\{V_p\}, \mathrm{span}\{\hat{V}^{(i)}\}\right\} \) denote the distance between the subspaces spanned by \( V_p \) and \( \hat{V}^{(i)} \).Then,
    \[
    \left\|E^{(i)}\right\| \le \left(\frac{\gamma_{p+1}^{(i)}}{\gamma_p^{(i)}}\right) \left\|E^{(i-1)}\right\|, \quad 
    \epsilon^{(i)} \le \prod_{j=1}^i\left(\frac{\gamma_{p+1}^{(j)}}{\gamma_p^{(j)}} \right) \left\|E^{(0)}\right\|.
    \]
    \item Let \( \beta^{(i)} = \left\|\rho^{(i)}(A)A(I - \rho^{(i)}(A))\right\| \). Assume the eigenvalues of \( A \) in \([a, b]\) are all simple. Define the spectral gap
    \[
    \mathrm{gap}_l = \min_{j \neq l} \left|\lambda_l - \lambda_j\right|, \quad l = 1, \dots, e.
    \]  
    Then, for each \( l \),
    \[
    \begin{aligned}
    \sin \angle \left(\hat{v}_l^{(i)}, v_l\right) &\leq \sqrt{1 + \frac{\left(\beta^{(i)}\right)^2}{\left(\mathrm{gap}_l\right)^2}} \cdot \prod_{j = 1}^i\left(\frac{\gamma^{(j)}_{p+1}}{\gamma^{(j)}_l}\right) \left\|E^{(0)}\right\|, \\
    \left|\hat{\lambda}_l^{(i)} - \lambda_l\right| &\leq \|A - \lambda_l I\| \sin^2 \angle \left(\hat{v}_l^{(i)}, v_l\right).
    \end{aligned}
    \]  
\end{enumerate}
\end{lemma}
 
Lemma \ref{lem:convergence} shows that the error term \( \|E^{(i)}\| \) at the \( i \)-th iteration is governed by a contraction factor \( |\gamma_{p+1} / \gamma_p| \), which primarily depends on the gap between the \( p \)-th and \( (p+1) \)-th eigenvalues of \( A \). 
Furthermore, it extends this analysis by introducing a contraction factor \( |\gamma_{p+1} / \gamma_l| \) to control the convergence of the \( l \)-th eigenpair.



\subsection{Polynomial Filtering}\label{Filtering}


The idea of the filter is to construct a function $s(x)$, such that values outside the region of interest are close to 0. 
For interior eigenvalue problem \eqref{ori_problem}, a suitable choice is the step function~\cite{jia2023feast} 
\begin{equation}\label{step_func}
    s(x) = \begin{cases}
            1,& x\in (a,b),\\
            \frac{1}{2},& x\in\{a, b\},\\
            0,& x\in [\lambda_{\text{min}},\lambda_{\text{max}}] \setminus [a,b].
            \end{cases}
\end{equation}
Note that $s(A)$ is a spectral projector corresponding to the eigenvalues of interest. 
Without loss of generality, let $[\lambda_{\text{min}},\lambda_{\text{max}}] = [-1, 1]$. 
This ensures that the Chebyshev expansion 
\[
s(x)=\lim_{k \rightarrow\infty} \rho_k(x) =\lim_{k \rightarrow\infty} \sum_{j=0}^{k} c_jT_j(x),
\]
where the Fourier coefficients are given by
\begin{equation}\label{c_j}
    c_j=\begin{cases}\frac1\pi(\arccos(a)-\arccos(b)),& j=0,\\
    \frac2\pi(\frac{\sin(j\arccos(a))-\sin(j\arccos(b))}j),& j>0,
    \end{cases}
\end{equation}
and $T_j(x)$ denotes the $j$-th Chebyshev polynomial.
In the general case, we replace $A$ with $l(A)$, where the linear transformation is defined as
\begin{equation*}
    l(x) = \frac{2x - \lambda_n - \lambda_1}{\lambda_n - \lambda_1} = l_1 x + l_2.
\end{equation*} 

As $k \to \infty$, the sequence $\rho_k(x)$ converges pointwise to $s(x)$. Theorem \ref{theo:fourier_error_bound} provides an upper bound on the convergence. However, the truncated series suffers from Gibbs oscillations at the discontinuities \cite{di2016efficient}. To suppress these oscillations, we introduce damping factors $d_{j,k}$, which depend on the polynomial degree $k$ and the index $j$. Incorporating these factors yields the damped expansion
\begin{equation*}
    \rho^D_k(x)=\sum_{j=0}^kc_jd_{j,k}T_j(x).
\end{equation*}
For example, for $j > 0$, the Jackson damping factors \cite{rivlin1981introduction, jia2023feast} are
\begin{equation*}
    d^J_{j,k}=\frac{(k+2-j)\sin(\frac{\pi}{k+2})\cos(\frac{j\pi}{k+2})+\cos(\frac{\pi}{k+2})\sin(\frac{j\pi}{k+2})}{(k+2)\sin\frac{\pi}{k+2}}.
\end{equation*} 
A simpler form, introduced by Lanczos~\cite{lanczos1988applied}, is given by
\begin{equation*}
    d^L_{j,k} = \frac{\sin\left(\frac{j\pi}{k+1}\right)}{\frac{j\pi}{k+1}},
\end{equation*}
which are known as the Lanczos damping factors or $\sigma$-factors \cite{lanczos1988applied}. 
The generalized version (exponent-$m$ Lanczos damping factors) is defined as
\begin{equation}\label{m_Lanczos_damping}
d^m_{j,k} = \left(\frac{\sin\left(\frac{j\pi}{k+1}\right)}{\frac{j\pi}{k+1}}\right)^m,
\end{equation}
where $m \ge 0$. The strength of the damping can be controlled by adjusting $m$.

In our implementation, the polynomial filter coefficients are precomputed during the setup phase. 
The product $\rho_k^D(l(A)) \cdot X$ in line 3 of Algorithm \ref{alg:SI} is evaluated implicitly using Clenshaw's algorithm. 
Utilizing the backward recurrence relation of Chebyshev polynomials, Clenshaw's algorithm ensures backward stability \cite{smoktunowicz2002backward} 
and offers superior memory efficiency compared to standard forward accumulation. 
The details of the algorithm can be found in Algorithm~\ref{alg:Compute_filter}.




\begin{algorithm}[!h]
	\caption{Clenshaw's algorithm for computing $\rho^D_k(l(A)) \cdot X$.}
	\label{alg:Compute_filter}
	\renewcommand{\algorithmicrequire}{\textbf{Input:}}
	\renewcommand{\algorithmicensure}{\textbf{Output:}}
	\begin{algorithmic}[1]
		\REQUIRE Hermitian matrix $A$, polynomial coefficients $\{a_{j}\}_{j=0}^{k+1}$, 
        linear transformation factors $\{l_1, l_2\}$ and matrix $X$.
		\ENSURE Matrix $V$.
        \STATE{$Y = X$}
        \STATE{$W = a_{k} Y$}
        \STATE{$V = 2 l_1 A W + (2 l_2 + a_{k-1} / a_{k}) W$}
        \STATE{Swap $V$ and $W$}
        \FOR{$j = k-2, \dots, 1$}
            \STATE{$V = 2 l_1 A W + 2 l_2 W - V + a_j Y$}
            \STATE{Swap $V$ and $W$}
        \ENDFOR
        \STATE{$V = l_1 A W + l_2 W - V + a_0 Y$}
	\end{algorithmic}
\end{algorithm}

\section{FSI with Adaptive Polynomial Filtering}\label{Adaptive Polynomial Filtering}

In this section, we propose an adaptive polynomial filter and integrate it into the FSI framework. 
For Chebyshev-series polynomial filtering, the two key parameters are the polynomial degree and the damping factors. 
A higher degree improves filter accuracy at the expense of more MVs per iteration; 
the required degree grows as the interval of interest $[a,b]$ narrows relative to $[\lambda_{\min},\lambda_{\max}]$ or as eigenvalues cluster near the endpoints $a$ and $b$ (see Theorem~\ref{theo:spectral_proj_error_bound}). 
To mitigate Gibbs oscillations \cite{di2016efficient}, we employ the exponent-$m$ Lanczos damping \eqref{m_Lanczos_damping}.

A common practice is to fix the filter before the iteration; an alternative is to update it adaptively \cite{winkelmann2019chase,zhang2024chebyshev}. A representative adaptive method is ChASE \cite{winkelmann2019chase}, which adjusts the polynomial degree per basis vector for extremal eigenproblems. For interior eigenproblems, however, the filter must approximate an indicator function of a spectral interval, and this strategy does not directly extend. 

\subsection{Adaptive Polynomial Filter}

We develop a new adaptive filter based on the Chebyshev expansion of a step function. This filter is suitable for most projection-based solvers for interior eigenvalue problems, such as the FSI method (Algorithm \ref{alg:SI}) and the Thick-Restart Lanczos method (TRLan) \cite{li2016thick}. 
To concretely demonstrate this adaptive mechanism, we implement it in the FSI framework. Further details are given in Subsection~\ref{Implementations}.



In iteration $i$, we choose the filter function 
\begin{equation}\label{AdaPolySI}
    \rho^m_{k^{(i)}, k}(x) = \sum_{j=0}^{k^{(i)}} c_j d^m_{j,k} T_j(l(x)),\quad k^{(i)} \le k,
\end{equation}
where \(k\) denotes the maximum polynomial degree selected during the setup phase, and \(k^{(i)}\) denotes the degree used at iteration \(i\). The Fourier coefficients $\{c_j\}$ and damping factors $\{d^m_{j,k}\}$ are computed once for the chosen $k$. 

By Lemma \ref{lem:convergence}, the filter \( P^{(i)} := \rho^m_{k^{(i)}, k}(A) \) should be designed to ensure \( |\gamma_{p+1}^{(i)} / \gamma_e^{(i)}| \) to be sufficiently small. However, this requires prior knowledge of the exact eigenvalues \( \lambda_e \) and \( \lambda_{p+1} \) of \( A \), which are generally unavailable in advance. In our algorithm, we propose an approximate approach using the Ritz values \( \hat{\lambda}_p^{(i-1)} \) and \( \hat{\lambda}_{e^{(i-1)}}^{(i-1)} \) obtained from iteration $i-1$, where $e^{(i-1)}$ denotes the number of Ritz values in the interval of interest at iteration $i-1$. Specifically, the filter update criterion is set to
\[
\left|\frac{\rho_{k^{(i-1)}, k}^m(\hat{\lambda}_p^{(i-1)})}{\rho_{k^{(i-1)}, k}^m(\hat{\lambda}_{e^{(i-1)}}^{(i-1)})}\right| < \tau_a,
\]
where $\tau_a$ denotes the adaptive threshold.

In practice, the polynomial degree \( k^{(i)} \) is typically chosen to be smaller than the initial degree \( k \). The specific strategy for selecting \( k^{(i)} \) dynamically is detailed in Algorithm \ref{alg:Determine_degree}. By iteratively adjusting \( k^{(i)} \), the algorithm maintains a lower computational cost while ensuring the filter’s effectiveness in suppressing non-target eigenvalues. 

\begin{algorithm}[!h]
	\caption{Compute the polynomial degree of the filter at the $i$-th iteration.}
        \label{alg:Determine_degree}
	\renewcommand{\algorithmicrequire}{\textbf{Input:}}
	\renewcommand{\algorithmicensure}{\textbf{Output:}}
	\begin{algorithmic}[1]
		\REQUIRE Ritz values $\hat\lambda^{(i-1)} = [\hat\lambda_1^{(i-1)}, \cdots, \hat\lambda_p^{(i-1)}]$, estimated eigenvalue count $e^{(i-1)}$ in $[a,b]$, polynomial coefficients $\{a_{j}\}_{j=0}^{k}$, linear transformation $l(x)$, threshold $\tau_a$
		\ENSURE Polynomial degree $k^{(i)}$ for the next iteration filter
            \STATE{$h_0 = 1$}
            \STATE{$h_1 = l(\hat\lambda^{(i-1)})$}
            \STATE{$\hat{\gamma}^{(i-1)} = a_{0} h_0$}
            \FOR{ $j = 1, 2, \cdots, k$ }
                \STATE{$\hat{\gamma}^{(i-1)} := \hat{\gamma}^{(i-1)} + a_j h_1$} 
                \STATE{Sort $\hat{\gamma}^{(i-1)}$ descending such that $\left|\hat{\gamma}_1^{(i-1)}\right| \geq \cdots \geq \left|\hat{\gamma}_p^{(i-1)}\right|$}
                \IF{ $|\hat{\gamma}_p^{(i-1)} / \hat{\gamma}_{e^{(i)}}^{(i-1)}| < \tau_a$ }
                    \STATE{$k^{(i)} = j$}
                    \STATE{\textbf{return} $k^{(i)}$}
                \ENDIF
                \STATE{$h_{\text{next}} = 2 l(\hat\lambda^{(i-1)})\cdot h_1 - h_0$} 
                \STATE{$h_0 = h_1$; $h_1 = h_{\text{next}}$}
            \ENDFOR
	\end{algorithmic}
\end{algorithm}

\subsection{Accuracy of the filter} \label{Accuracy of the filter}

By Lemma \ref{lem:convergence}, a sufficient condition for convergence of the filtered subspace iteration is that $|\gamma^{(i)}_{e}| > |\gamma^{(i)}_{p+1}|$ at each iteration $i$.
To identify the quantities influencing convergence, we derive a pointwise convergence bound for the filtering function.

Let $g^m_{k^{(i)}, k}(\theta) = \rho^m_{k^{(i)}, k}(\cos\theta)$ and $g(\theta) = \rho(\cos\theta)$. Since $T_j(x) = \cos(j \arccos x)$, we have
\begin{equation}\label{Fourier_series_damping}
g^m_{k^{(i)}, k}(\theta) = \sum_{j=0}^{k^{(i)}} c_j d^m_{j,k} \cos(j\theta), \quad \theta \in [0, \pi].
\end{equation}

Let $\alpha=\arccos a$ and $\beta=\arccos b$. Then $g^m_{k^{(i)}, k}(\theta)$ is the damped partial sum of the Fourier series of $g(\theta)$, defined as
\begin{equation*}
g(\theta) = \begin{cases}
1, & \theta\in (\beta, \alpha),\\
\frac{1}{2}, & \theta\in\{\alpha, \beta\},\\
0, & \theta\in [0,\pi]\setminus [\beta, \alpha].
\end{cases}
\end{equation*}

The following lemma provides a pointwise bound for the Fourier partial sums of \(\frac{\pi-x}{2}\).
\begin{lemma} \label{lem:Rk_bound}
    Let $S_k(x) = \sum_{j=1}^k \frac{\sin(jx)}{j}$ be the $k$-th partial sum of the Fourier series of $f(x) = \frac{\pi - x}{2}$ for $x \in (0, 2\pi)$, and $R_k(x) = \frac{\pi - x}{2} - S_k(x)$. For any $k \ge 1$ and $x \in (0, \pi]$, we have
    \begin{equation*} 
        |R_k(x)| \le \frac{1}{(k+1)\sin(x/2)}.
    \end{equation*}
\end{lemma}

\begin{proof}
    Applying Abel's summation formula,
    we obtain
    \[
        \sum_{j=k+1}^N \frac{\sin(jx)}{j} = \frac{A_N(x)}{N} - \frac{A_k(x)}{k+1} + \sum_{j = k+1}^{N-1} A_j(x) \left( \frac{1}{j} - \frac{1}{j+1} \right),
    \]
    where 
    \[
        A_j(x) = \sum_{m=1}^j \sin(mx) = \frac{\cos(x/2) - \cos((j+1/2)x)}{2\sin(x/2)}.
    \]
    Taking the limit as $N \to \infty$ leads to
    \[
    \begin{aligned} 
        2\sin(x/2) R_k(x)
        &= - \frac{A_k(x)}{k+1} + \sum_{j=k+1}^{\infty} \frac{A_j(x)}{j(j+1)} \\
        &= - \frac{\cos(x/2) - \cos((k+1/2)x)}{k+1} + \sum_{j=k+1}^{\infty} \frac{\cos(x/2) - \cos((j+1/2)x)}{j(j+1)} \\ 
        &= \frac{\cos((k+1/2)x)}{k+1} - \sum_{j=k+1}^\infty \frac{\cos((j+1/2)x)}{j(j+1)}.
    \end{aligned}
    \]
    Taking absolute values, we obtain
    \[
        2\sin(x/2) |R_k(x)| \le \frac{1}{k+1} + \sum_{j=k+1}^\infty \frac{1}{j(j+1)} = \frac{2}{k+1},
    \]
    which implies $|R_k(x)| \le \frac{1}{(k+1)\sin(x/2)}$.

\end{proof}

When $d^m_{j,k}=1$ for $j=0,\dots,k$, \eqref{Fourier_series_damping} reduces to the undamped Fourier partial sum of $g(\theta)$, denoted by $g_k(\theta)$. We obtain the following bound.

\begin{theorem}\label{theo:fourier_error_bound}

    Let $g_k(\theta) = \sum_{j=0}^k c_j \cos(j\theta)$ be the $k$-th partial sum of the Fourier series of $g(\theta)$. For any $k \ge 1$ and $\theta \in [0, \pi]$, we have
    \begin{equation}\label{gk_bound}
        |g(\theta) - g_k(\theta)| \le \frac{J(\theta)}{\pi (k+1)},
    \end{equation}
    where
    \begin{equation}\label{J_theta}
        J(\theta) =
        \begin{cases} 
        \frac{1}{|\sin \frac{\theta + \alpha}{2}|} + \frac{1}{|\sin \frac{\theta - \alpha}{2}|} + \frac{1}{|\sin \frac{\theta + \beta}{2}|} + \frac{1}{|\sin \frac{\theta - \beta}{2}|}, & \theta\neq\alpha,\beta,\\
        \frac{1}{|\sin\alpha|} 
        + 
        \frac{1}{\bigl|\sin\frac{\alpha+\beta}{2}\bigr|} 
        + 
        \frac{1}{\bigl|\sin\frac{\alpha-\beta}{2}\bigr|}, 
        & \theta=\alpha,\\
        \frac{1}{|\sin\beta|} 
        + 
        \frac{1}{\bigl|\sin\frac{\alpha+\beta}{2}\bigr|} 
        + 
        \frac{1}{\bigl|\sin\frac{\alpha-\beta}{2}\bigr|}, 
        & \theta=\beta. 
        \end{cases}
    \end{equation}
\end{theorem}

\begin{proof}

Since 
\[
  \begin{aligned}
        g(\theta) - g_k(\theta) 
        &=  \frac{1}{\pi} \bigg(
        \sum_{j=k+1}^{\infty} \frac{\sin j(\theta+\alpha)}{j} 
        - \sum_{j=k+1}^{\infty} \frac{\sin j(\theta-\alpha)}{j} \\
        &\qquad - \sum_{j=k+1}^{\infty} \frac{\sin j(\theta+\beta)}{j} 
        + \sum_{j=k+1}^{\infty} \frac{\sin j(\theta-\beta)}{j}\bigg)\\
        &\le \frac{1}{\pi}\left(|R_k(\theta+\alpha)| + |R_k(\theta-\alpha)| + |R_k(\theta+\beta)| + |R_k(\theta-\beta)|\right).
    \end{aligned}
\]
The conclusion follows naturally from Lemma \ref{lem:Rk_bound}.


\end{proof}

\begin{remark}\label{rem:J_theta}
    Note that, for $\theta \neq \alpha, \beta$,
    \[
        J(\theta) \le \frac{4\pi}{d(\theta)},  
    \]
    where $d(\theta) = \min\{|\theta \pm \alpha|, |\theta \pm \beta|\}$. Therefore, when $\theta$ is far from the discontinuities $\alpha$ and $\beta$, $J(\theta)$ is small, leading to a tighter bound in \eqref{gk_bound}. Conversely, as $\theta$ approaches $\alpha$ or $\beta$, $J(\theta)$ increases, indicating a looser bound. This behavior reflects the Gibbs phenomenon, where oscillations occur near discontinuities in the partial sums of a Fourier series.
\end{remark}

The following theorem gives a pointwise bound for the damped Fourier expansion \eqref{Fourier_series_damping}.
\begin{theorem}\label{theo:damped_fourier_error_bound}
Let $g^m_{k^{(i)},k}(\theta)$ be the damped Fourier partial sum defined in \eqref{Fourier_series_damping}. For any $k \ge 1$ and $\theta \in [0, \pi]$, we have
\begin{equation}\label{eq:damped_fourier_error_bound}
    |g(\theta) - g^m_{k^{(i)},k}(\theta)| \le \frac{d_{k^{(i)},k}^m\,J(\theta)}{\pi(k^{(i)}+1)} + \frac{m\mathcal{C}_m\,J(\theta)}{k+1} + \frac{m\pi(\pi-\theta)}{3(k + 1)^2},
\end{equation}
where $1 \le k^{(i)} \le k$, $m > 0$, $d_{k^{(i)},k}^m$ is defined in \eqref{m_Lanczos_damping}, $J(\theta)$ is defined in \eqref{J_theta}, and
\begin{equation}\label{C_m}
    \mathcal{C}_m := \int_0^{\pi} \frac{(\sin u)^{m-1}(\sin u - u\cos u)}{u^{m+2}}\,du.
\end{equation}
\end{theorem}

\begin{proof}

Using \eqref{c_j} and the identity $\sin u\cos v=\tfrac12\bigl(\sin(u+v)+\sin(u-v)\bigr)$, we can rewrite
\begin{equation}\label{eq:damped_error_split}
\begin{aligned}
g(\theta) - g^m_{k^{(i)},k}(\theta)
&= \sum_{j=1}^{k^{(i)}} \bigl(1-d^m_{j,k}\bigr) c_j \cos(j\theta)
+ \sum_{j=k^{(i)}+1}^{\infty} c_j \cos(j\theta)\\
&= \frac{1}{\pi}\bigl(E(\theta+\alpha) - E(\theta-\alpha) - E(\theta+\beta) + E(\theta-\beta)\bigr),
\end{aligned}
\end{equation}
where
\begin{equation}\label{eq:def_E}
E(x) := \sum_{j=1}^{k^{(i)}} \bigl(1-d^m_{j,k}\bigr)\frac{\sin(jx)}{j} + \sum_{j=k^{(i)}+1}^{\infty} \frac{\sin(jx)}{j}.
\end{equation}

Let $b_j(x):=\sin(jx)/j$. Then
\[
R_\ell(x):=\sum_{m=\ell+1}^{\infty} b_m(x)=\sum_{m=\ell+1}^{\infty}\frac{\sin(mx)}{m}
\]
and $b_j(x)=R_{j-1}(x)-R_j(x)$. \eqref{eq:def_E} can be rewritten as
\begin{equation}\label{eq:E_sbp}
E(x)=\bigl(1-d^m_{1,k}\bigr)R_0(x)
+\sum_{j=1}^{k^{(i)}-1}\bigl(d^m_{j,k}-d^m_{j+1,k}\bigr)R_j(x)
+d^m_{k^{(i)},k}R_{k^{(i)}}(x).
\end{equation}

For $x\in(0,2\pi)$, we have $R_0(x)=\sum_{m=1}^{\infty}\sin(mx)/m=(\pi-x)/2$. Moreover, by Lemma~\ref{lem:Rk_bound}, for all $j\ge 1$,
\[
|R_j(x)|\le \frac{1}{(j+1)|\sin(x/2)|}.
\]
Taking absolute values in \eqref{eq:E_sbp} gives
\begin{equation}\label{eq:E_bound_1}
|E(x)| \le \frac{(1-d^m_{1,k})(\pi-x)}{2}
+\frac{1}{|\sin(x/2)|}\sum_{j=1}^{k^{(i)}-1}\frac{d^m_{j,k}-d^m_{j+1,k}}{j+1}
+\frac{d^m_{k^{(i)},k}}{(k^{(i)}+1)|\sin(x/2)|}.
\end{equation}

Let $h:=\pi/(k+1)$. For the first term, using $\sin y\ge y-y^3/6$ for $y\in(0,\pi)$ and Bernoulli's inequality, we obtain
\begin{equation}\label{eq:d1_small}
1-d^m_{1,k}=1-\left(\frac{\sin h}{h}\right)^m \le \frac{m h^2}{6}.
\end{equation}


To bound the middle term, define $f(t):=\bigl(\sin(th)/(th)\bigr)^m$ for $t>0$, so that $d^m_{j,k}=f(j)$.
Since $f$ is positive and decreasing on $(0,\pi/h)$, we have
$f(j)-f(j+1)=\int_j^{j+1}|f'(t)|\,dt$, and the bound $\frac{1}{j+1}\le\frac{1}{t}$ for $t\in[j,j+1]$ gives
\begin{equation}\label{eq:d_diff_sum}
\sum_{j=1}^{k^{(i)}-1}\frac{d^m_{j,k}-d^m_{j+1,k}}{j+1}
\le \int_1^{k^{(i)}}\frac{|f'(t)|}{t}\,dt.
\end{equation}
Differentiating yields
\[
f'(t)= m\left(\frac{\sin(th)}{th}\right)^{m-1}\frac{th\cos(th)-\sin(th)}{t^2 h},
\]
so that
\[
\frac{|f'(t)|}{t}= m\,\frac{(\sin(th))^{m-1}\bigl(\sin(th)-th\cos(th)\bigr)}{(th)^{m+1}\,t\,h}.
\]
Substituting $u=th$ into \eqref{eq:d_diff_sum}, we obtain
\begin{equation}\label{eq:d_diff_integral}
\int_1^{k^{(i)}}\frac{|f'(t)|}{t}\,dt
= m\int_{h}^{k^{(i)}h}\frac{(\sin u)^{m-1}(\sin u-u\cos u)}{u^{m+2}}\,du
\le m\,\mathcal{C}_m,
\end{equation}
where $\mathcal{C}_m$ is defined in \eqref{C_m}. 
The integrand is non-negative and integrable on $(0,\pi)$.

Substituting \eqref{eq:d1_small} and \eqref{eq:d_diff_integral} into \eqref{eq:E_bound_1} and recalling $h=\pi/(k+1)$ yields
\begin{equation*} 
|E(x)| \le \frac{m\pi^2(\pi-x)}{12(k+1)^2}
+\frac{m\pi\,\mathcal{C}_m}{(k+1)|\sin(x/2)|}
+\frac{d^m_{k^{(i)},k}}{(k^{(i)}+1)|\sin(x/2)|}.
\end{equation*}

Finally, by \eqref{eq:damped_error_split} and the triangle inequality, we obtain
\[
\begin{aligned}
|g(\theta)-g^m_{k^{(i)},k}(\theta)|
&\le \frac{1}{\pi}\Bigl(|E(\theta+\alpha)| + |E(\theta-\alpha)| + |E(\theta+\beta)| + |E(\theta-\beta)|\Bigr)\\
&\le \frac{d^m_{k^{(i)},k}J(\theta)}{\pi(k^{(i)}+1)}
+\frac{m\mathcal{C}_m\,J(\theta)}{k+1}
+\frac{m\pi(\pi-\theta)}{3(k+1)^2}.
\end{aligned}
\]
\end{proof}

Note that, when $m=0$, both $m\mathcal{C}_m$ and $m\pi(\pi-\theta)/3$ vanish and $d^0_{k^{(i)},k}=1$, so \eqref{eq:damped_fourier_error_bound} reduces to the undamped bound \eqref{gk_bound}. Moreover, for fixed $m>0$, the bound \eqref{eq:damped_fourier_error_bound} is $O(J(\theta)/(k^{(i)}+1))$ uniformly in $k^{(i)}\le k$.

We illustrate the bound \eqref{eq:damped_fourier_error_bound} with a simple example. Let $a = 0.1$, $b = 0.6$, $m = 0.5$, and $k = 2.5k^{(i)}$. In the left panel of Fig.~\ref{fig:bound_verification}, we fix $x=\cos\theta=0.3$ and increase $k^{(i)}$; in the right panel, we fix $k=100$ and vary $x$ from $-0.1$ to $0.3$, comparing the error with the corresponding bound.

For $x:=\cos\theta=0.3$, as $k^{(i)}$ increases, the upper envelope of the error and the bound in \eqref{eq:damped_fourier_error_bound} become nearly parallel and remain of the same order of magnitude. This indicates that the estimate \eqref{eq:damped_fourier_error_bound} accurately tracks the decay of the error. For smaller $k^{(i)}$, the error is more strongly affected by the term involving $J(\theta)$; by Remark \ref{rem:J_theta}, this contribution is of order $d(\theta)^{-1}$. Thus, for modest degrees, the error is governed primarily by the distance from the evaluation point to the discontinuity, consistent with the Gibbs phenomenon. 
The right panel also illustrates that the error increases as $x$ approaches the discontinuity point $a$, consistent with the bound \eqref{eq:damped_fourier_error_bound}.

\begin{figure}[htbp]
    \centering
    \includegraphics[width=1\linewidth]{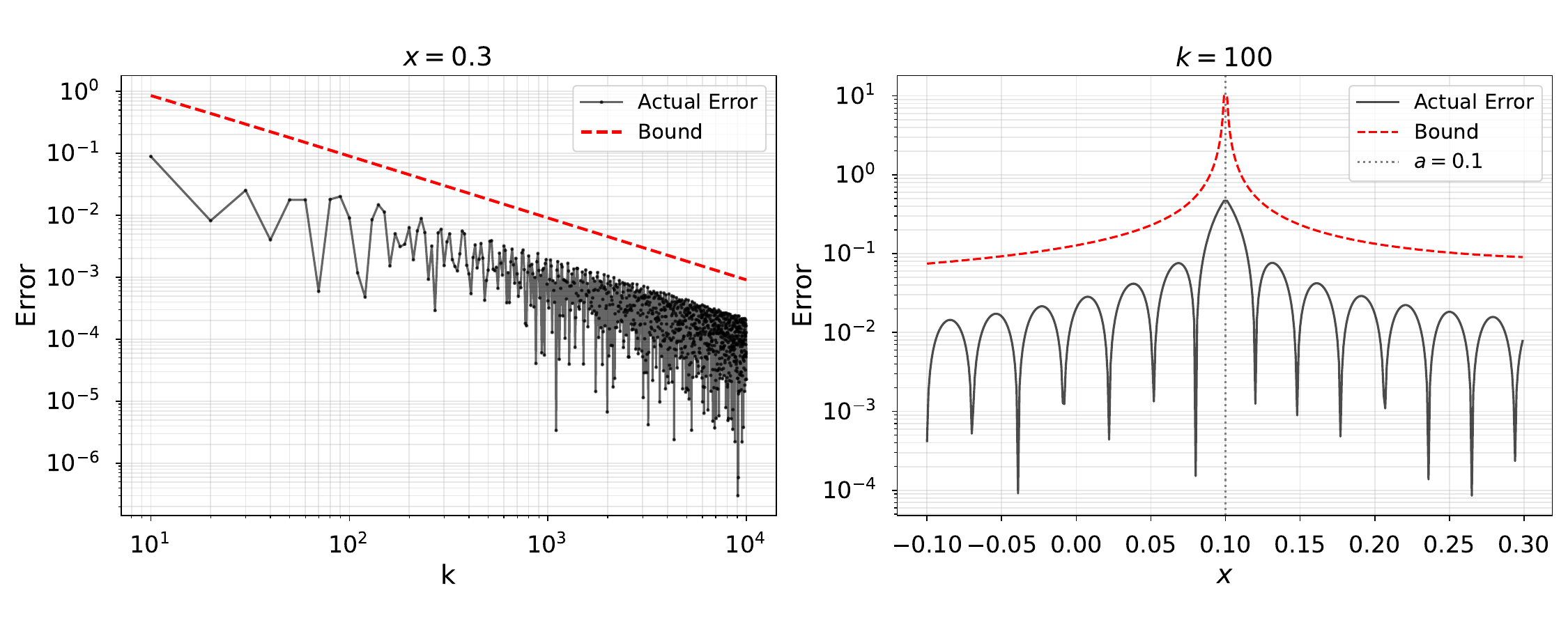}
        \caption{Errors and error bounds for \eqref{eq:damped_fourier_error_bound}.}
    \label{fig:bound_verification}
\end{figure}

By Theorem~\ref{theo:damped_fourier_error_bound}, we further obtain a bound on $\|P - P^{(i)}\|_2$.
\begin{theorem}\label{theo:spectral_proj_error_bound}

Let $P = s(A)$ be a spectral projector onto the invariant subspace associated with the eigenvalues of $A$ in $[a, b]$, and let $P^{(i)} = \rho^m_{k^{(i)}, k}(A)$ be the approximate spectral projector constructed at iteration $i$. Then, for any $1 \le k^{(i)} \le k$, we have
\begin{equation}\label{spectral_proj_error_bound}
    \|P - P^{(i)}\|_2 \le \frac{d_{k^{(i)},k}^m\,J^*}{\pi(k^{(i)}+1)} + \frac{m\mathcal{C}_m\,J^*}{k+1} + \frac{m\pi^2}{3(k + 1)^2},   
\end{equation}
where $J^* = \max_{j=1,\cdots,n} J(l(\arccos\lambda_j))$, $J(\phi)$ is defined in \eqref{J_theta}, and $\mathcal{C}_m$ is defined in \eqref{C_m}.
\end{theorem}

\begin{proof}
    By Theorem~\ref{theo:damped_fourier_error_bound}, we have
    \[
    \begin{aligned}
        \|P - P^{(i)}\|_2 
        &= \|\rho(A) - \rho^m_{k^{(i)}, k}(A)\|_2\\
        &= \max_{j=1,\cdots,n} | \rho(\lambda_j) - \rho^m_{k^{(i)}, k}(\lambda_j) |\\
        &= \max_{j=1,\cdots,n} | g(\phi_j) - g^m_{k^{(i)}, k}(\phi_j) |\\
        &\le \frac{d_{k^{(i)},k}^m\,J^*}{\pi(k^{(i)}+1)} + \frac{m\mathcal{C}_m\,J^*}{k+1} + \frac{m\pi^2}{3(k + 1)^2},
    \end{aligned}
    \]
    where $\phi_i = l(\arccos\lambda_i)$. 
    The conclusion follows directly.
\end{proof}

By the definition of the spectral projector, the eigenvalues of $P$ belong to $\{1,0\}$, with the additional value $1/2$ only in the exceptional case where $a$ or $b$ is an eigenvalue of $A$. Therefore, if $a$ and $b$ are not eigenvalues of $A$ and $\|P-P^{(i)}\|_2\le \tfrac12$, then
\[
    |\gamma_1^{(i)}| \ge \cdots \ge |\gamma_{e^{(i)}}^{(i)}| > \frac12 > |\gamma_{e^{(i)} + 1}^{(i)}| \ge \cdots \ge |\gamma_n^{(i)}|,
\]
and hence it suffices to enforce
\[
\frac{d_{k^{(i)},k}^m\,J^*}{\pi(k^{(i)}+1)} + \frac{m\mathcal{C}_m\,J^*}{k+1} + \frac{m\pi^2}{3(k + 1)^2} \le \frac12.
\]

In practice, $k$ is chosen large enough so that the last two terms in \eqref{spectral_proj_error_bound} are negligible, and the adaptive degree $k^{(i)}\le k$ is kept as small as possible. Since $0 < d_{k^{(i)},k}^m \le 1$, the first term on the left-hand side of \eqref{spectral_proj_error_bound} dominates. A sufficient condition for $\|P - P^{(i)}\|_2 < \tfrac12$ is therefore
\[
\frac{J^*}{\pi(k^{(i)}+1)} < \frac12,
\quad\text{or}\quad
k^{(i)} > \frac{2J^*}{\pi} - 1.
\]

The quantity \(J^*\) depends on the spectrum and is not available a priori. By definition \eqref{J_theta}, \(J(\theta)\) is bounded by the reciprocal of the distance from \(\theta\) to the nearest discontinuity (cf.\ Remark~\ref{rem:J_theta}); consequently, \(J^*\) scales as \(e/(\alpha-\beta)\) when the \(e\) eigenvalues inside \([a,b]\) are roughly uniformly distributed. Motivated by this scaling, we set the initial degree as
\begin{equation}\label{compute_k_1}
    k^{(1)} = \left\lceil\frac{C}{\alpha - \beta}\right\rceil - 1,
\end{equation}
where $C > \alpha - \beta$ is a user-specified constant. For $i > 1$, \(k^{(i)}\) is refined using the spectral information gathered during the iteration via Algorithm~\ref{alg:Determine_degree}.




\subsection{Algorithm and Implementation Details}\label{Implementations}

This subsection presents the proposed improvements to the algorithmic framework and discusses several implementation details. The pseudo-code for the full version of the algorithm of this paper is shown in Algorithm \ref{alg:ACJFEAST}.

\begin{algorithm}[!htbp]
	\caption{FSI with adaptive polynomial filtering for Hermitian eigenproblems.}
	\label{alg:ACJFEAST}
	\renewcommand{\algorithmicrequire}{\textbf{Input:}}
	\renewcommand{\algorithmicensure}{\textbf{Output:}}
	\begin{algorithmic}[1]
		\REQUIRE The $n$-by-$n$ Hermitian matrix $A$, the initial degree of Chebyshev filter $k,~k^{(1)}$, the interval of interest $[a, b]$ and the thresholds $\tau_a,~\tau_c$.
		\ENSURE The eigen-pairs ($\Lambda, V$)  of matrix $A$ in the interval $[a, b]$. 

        \STATE\label{ACJFEAST-Compute_CJ_cof}{Compute $\{c_{j}\}_{j=0}^k,~\{d_{j,k}\}_{j=0}^k,~\{l_1,l_2\}$ and $p$.}
        \STATE\label{ACJFEAST-init}{Select an n-by-p order orthogonal matrix $\hat V^{(0)}$. Set $\Lambda = [\ ],\ V = [\ ],\ n_{\text{lock}} = 0,\ n_{\text{check}} = 0$.}
		\FOR{$i = 1, 2,\cdots$}
		\STATE\label{ACJFEAST-Approximate_basis_mat}{Filtering (using Algorithm \ref{alg:Compute_filter}): $\hat X^{(i)} = \rho^m_{k^{(i)},k}(A) \hat V^{(i-1)}$.}
		\STATE\label{ACJFEAST-orth}{Orthogonalization (using Cholesky QR \cite{yamazaki2015mixed}): $ \hat Q^{(i)} = \text{orth}([V, \hat X^{(i)}])$. }
        \STATE{Set $Q^{(i)} = \hat Q^{(i)}(:,\ 1:p-n_{\text{lock}})$.}
        \STATE\label{ACJFEAST-RR}{[$\hat V^{(i)}$, $\hat\Lambda^{(i)}$] = \text{Rayleigh-Ritz}($A$,  $Q^{(i)}$).}
        \STATE{\(\mathcal{I}^{(i)} = \left\{ j : \hat\Lambda^{(i)}_{j,j} \in [a,b]\right\},~e^{(i)} = \left|\mathcal{I}^{(i)}\right|\).}
        \IF{$e^{(i)} $ equal to 0}
        \STATE{Proceed to the next iteration.}
        \ENDIF
        \STATE\label{ACJFEAST-adaptive_degree}{Obtain the filter degree $k^{(i+1)}$ using Algorithm \ref{alg:Determine_degree}.}
        \STATE\label{ACJFEAST-compute_res}{Compute the residuals of the eigenpairs indexed by \(\mathcal{I}^{(i)}\), denoted as $\|r_j\|_2,\ j = 1,\cdots, e^{(i)}$.}
        \STATE\label{ACJFEAST-check}{Compute $\tau_s^{(i)}$ by \eqref{tau_s}.}
        \STATE\label{ACJFEAST-check2}{\(\mathcal{C}^{(i)} = \left\{ j\in\mathcal{I}^{(i)} : \|r_j\|_2 < \tau_s^{(i)}\right\}\), $n_{\text{check}}^{(i)} = |\mathcal{C}^{(i)}|,~
        n_{\text{check}}^{(i)} += n_{\text{lock}}$.}
        \IF{$n_{\text{check}}^{(i)} > 0$ and equal to $n_{\text{check}}^{(i-1)}$}
        \STATE{$\mathcal{I}^{(i)} = \mathcal{C}^{(i)}$.}
        \ENDIF
        \STATE\label{ACJFEAST-lock1}{\(\mathcal{L}^{(i)} = \left\{ j\in\mathcal{I}^{(i)} : \|r_j\|_2 < \tau_c\right\},\) $n_{\text{lock}}^{(i)} = |\mathcal{L}^{(i)}|$.}
        \STATE\label{ACJFEAST-lqock2}{$\Lambda := \text{diag}(\Lambda,~\hat\Lambda^{(i)}_{(\mathcal{L}^{(i)},~\mathcal{L}^{(i)})}),\ V := [V,~V^{(i)}_{(:,~\mathcal{L}^{(i)})}]$.}
        \STATE\label{ACJFEAST-lock3}{$\hat V^{(i)} := \hat V^{(i)}_{(:,~(\mathcal{L}^{(i)})^c)},~n_{\text{lock}} += n_{\text{lock}}^{(i)}$.}
        \IF{$\mathcal{I}^{(i)}\setminus \mathcal{L}^{(i)}$ is empty}
        \STATE{Return $(\Lambda, V)$.}
        \ENDIF
		\ENDFOR  
	\end{algorithmic}
\end{algorithm}




\begin{algorithm}[!h]
\caption{MaSpMM Algorithm~\cite{MaSpMM}}
\label{algo:mespmm}
\renewcommand{\algorithmicrequire}{\textbf{Input:}}
\renewcommand{\algorithmicensure}{\textbf{Output:}}
\begin{algorithmic}[1]
\REQUIRE Sparse matrix $A[M][N]$, dense matrix $B[N][K]$.
\ENSURE Dense matrix $C[M][K]$.
\FOR{$k = 0$ to $\lceil K/T_k \rceil - 1$}
    \STATE $kbound = \min((k+1)\cdot T_k, K-1)$
    \FOR{$i = 0$ to $\lceil M/T_i \rceil - 1$ \textbf{in parallel}}
        \FOR{$jj = \textit{act\_colSeg}[i]$ to $\textit{act\_colSeg}[i+1]$}
            \STATE $col = col\_idx[jj]$
            \FOR{$kk = k \cdot T_k$ to $kbound$ step $vecLen$}
                \STATE $rowSeg\_B = \texttt{load\_vector}(B[col][kk:kk+vecLen])$
                \FOR{$ii = colSeg\_ptr[jj]$ to $colSeg\_ptr[jj+1]$}
                    \STATE $row = row\_idx[ii]$, $val = values[ii]$
                    \STATE $A\_broadcast = \texttt{broadcast}(val)$
                    \STATE $rowSeg\_C = \texttt{load\_vector}(C[row][kk:kk+vecLen])$
                    \STATE $rowSeg\_C \mathrel{+}= A\_broadcast \times rowSeg\_B$
                    \STATE \texttt{store\_vector}$(C[row][kk:kk+vecLen], rowSeg\_C)$
                \ENDFOR
            \ENDFOR
        \ENDFOR
    \ENDFOR
\ENDFOR
\end{algorithmic}
\end{algorithm}
\subsubsection{Efficient SpMM Implementation}
\label{sec:SPMM}
Sparse matrix--dense matrix multiplication (SpMM) is a core operator in many high-performance computing and graph applications.
In our filtered subspace iteration framework, SpMM dominates the cost of the filtering step (line~\ref{ACJFEAST-Approximate_basis_mat} of Algorithm~\ref{alg:ACJFEAST}), where the polynomial filter $\rho(A)X$ is evaluated via Clenshaw's recurrence involving repeated multiplications of $A$ by a dense block of width $K$.
Given a sparse matrix $A\in\mathbb{R}^{M\times N}$ and a dense matrix $B\in\mathbb{R}^{N\times K}$, SpMM computes $C = AB$ with $C\in\mathbb{R}^{M\times K}$.
In practice, $A$ is commonly stored in compressed sparse row (CSR) format~\cite{saad2011numerical}, while $B$ and $C$ are typically stored in a dense row-major layout.
Conventional CSR SpMM suffers from poor cache reuse as $K$ grows: for each row of $A$, it repeatedly loads the same row segments of $B$ from memory.
Representative tile-based SpMM such as J-Stream~\cite{jstream} partitions the sparse matrix $A$ into row blocks of size $T_i$ and stores nonzeros within each row block in column-major order (column-segments).
This layout confines the access range of each column-segment to the dense matrix $C$ to the segment length, enabling reuse of the $C$ tile in cache.
However, J-Stream keeps the dense operands $B$ and $C$ in the conventional row-major layout; within each $C$-tile, accesses to the corresponding row segments become non-contiguous, degrading SIMD efficiency and cache utilization as $K$ increases~\cite{jstream}.
Fig.~\ref{fig:tilespmm} illustrates the storage formats and workflows of CSR-SpMM and J-Stream.

\begin{figure}[tb]
    \setlength{\abovecaptionskip}{2pt} 
    \setlength{\belowcaptionskip}{0pt} 
    \centering
    \includegraphics[width=0.9\linewidth]{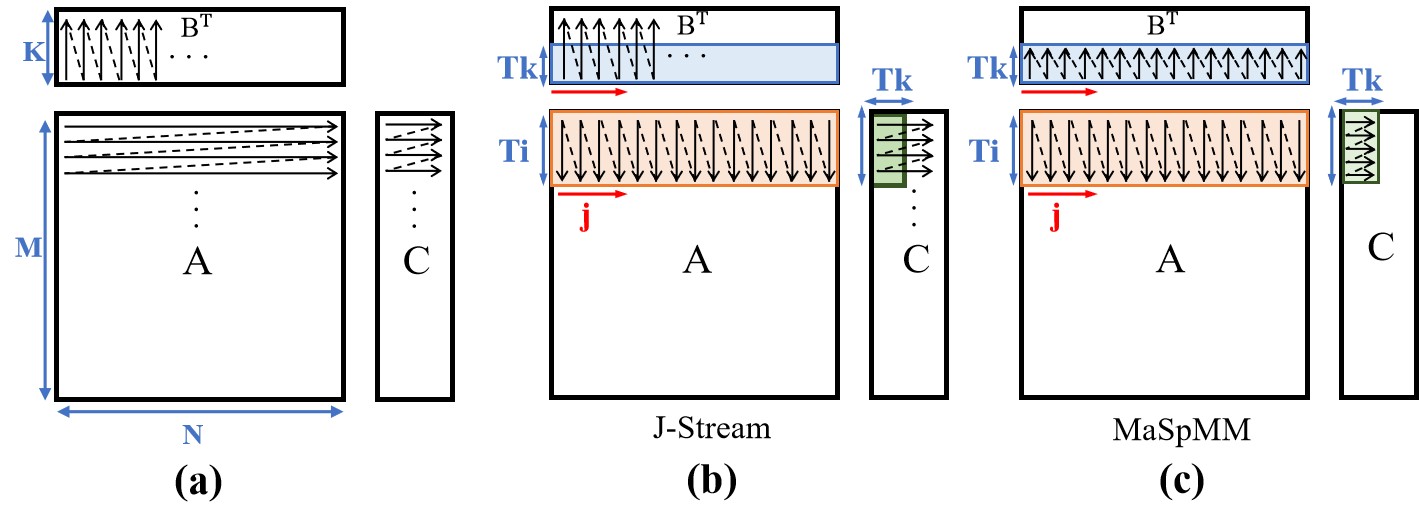}
    \caption{(a) The CSR SpMM;  (b) The J-Stream workflow. The \emph{column-segment} format is used in this work, in which each column-segment within a row block of $\mathbf{A}$ is stored contiguously; (c) The MaSpMM workflow. The \emph{row-segment} format is further adopted in this work, in which the dense matrix is partitioned into column blocks, and each row-segment within a column block is stored contiguously.}
    \label{fig:tilespmm}
\end{figure}

We use MaSpMM, a memory-aware SpMM microkernel designed to sustain high throughput as the dense dimension $K$ grows~\cite{MaSpMM}.
MaSpMM addresses the dense-side limitation by adopting a \emph{row-segment} layout: the dense matrix is partitioned into column blocks of width $T_k$, and within each column block the operands are stored in contiguous row segments.
By tuning $T_i$ and $T_k$, the $T_i\times T_k$ output $C$-tile can remain cache-resident during accumulation.
Algorithm~\ref{algo:mespmm} outlines the MaSpMM workflow.
The first-level loop iterates over dense column blocks (each of width $T_k$), and the second-level loop parallelizes over sparse row blocks (each of height $T_i$), so that each thread computes one $T_i\times T_k$ tile of $C$ and keeps it in cache.
Within each row block, nonzeros are organized as column-segments: each segment groups nonzeros that share the same column index of $A$ (Line 5), so that the corresponding row of $B$ can be loaded once and reused for all nonzeros in that segment (Line 7).
For each column-segment (column index \texttt{col}), MaSpMM loads the corresponding row segment of $B[\texttt{col}]$ once and reuses it to update the affected rows of the current $T_i\times T_k$ tile in $C$ via SIMD FMA (Lines 8-13).
This segment-oriented reuse improves dense-side locality and keeps the $C$ tile cache-resident during accumulation, helping MaSpMM sustain stable throughput as $K$ grows.

\subsubsection{Spurious Eigenvalues}

Algorithm \ref{alg:SI} adopts a simple convergence criterion: the iteration terminates once all Ritz values in the interval of interest are converged.
While this criterion is generally reliable, it is often overly conservative in practice.

To illustrate this issue, we consider the matrix \texttt{VCNT4000}\footnote{The matrix is taken from the ELSES library \cite{eles-matrix}, which consists of matrices arising from quantum mechanical problems.}  with the interval of interest set to \([-0.12,0.01]\).
The filter \eqref{AdaPolySI} is applied directly within Algorithm \ref{alg:SI}, and the iteration is terminated either when the relative residuals fall below \(10^{-10}\) or after at most 20 iterations.
Fig. \ref{fig:spurious_convergence} reports the distribution of unconverged Ritz values and the residuals of Ritz values inside the interval. 
As shown in Fig. \ref{fig:spurious_convergence}(a), a large number of Ritz values in the interval remain unconverged throughout the iteration. 
Fig. \ref{fig:spurious_convergence}(b) further indicates that the residuals of these Ritz pairs stay at the level of \(10^{-1}\). 

\begin{figure}[htbp!]
    \centering
     \includegraphics[width=0.8\linewidth]{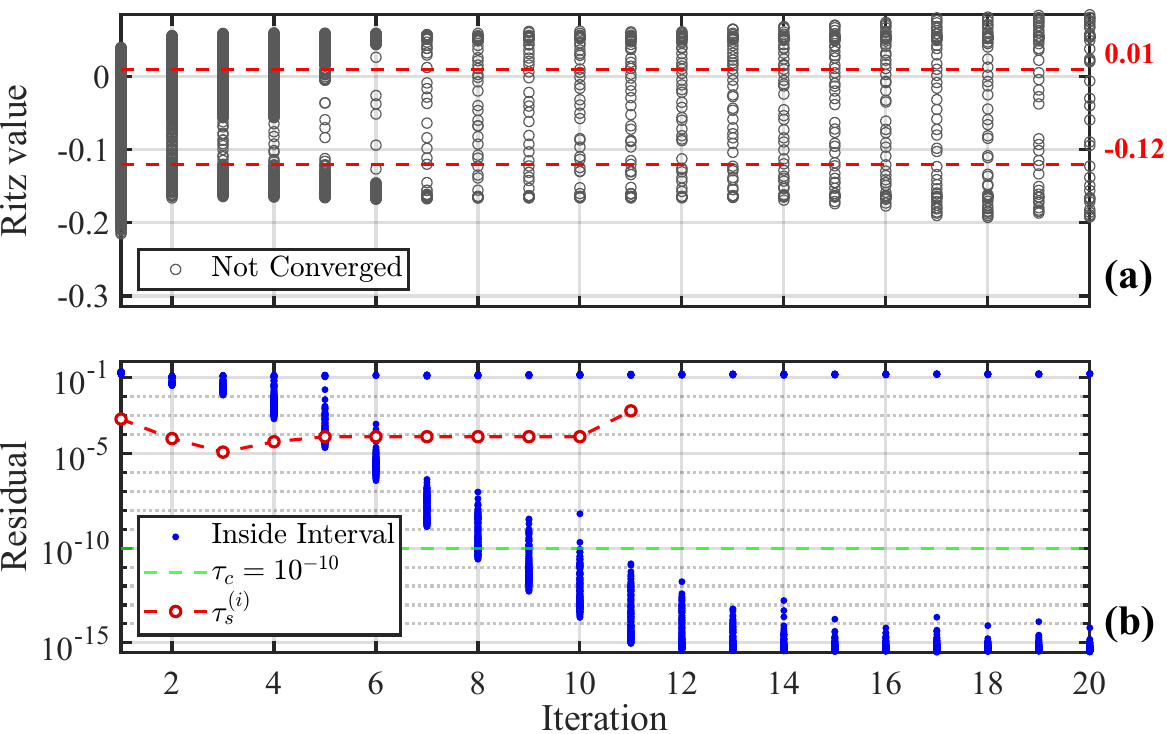}
    \caption{Ritz values and residuals for the VCNT4000 with interval $[-0.12,0.01]$.}
    \label{fig:spurious_convergence}
\end{figure}

In fact, all target eigenpairs have already satisfied the prescribed accuracy at the $11$-th iteration. However, the presence of spurious Ritz values inside the interval prevents the convergence criterion from being met. 


Several rational filtering-based methods attempt to estimate an accurate value of $e$ either prior to the iteration or during the iterative process \cite{peter2014feast, yin2019contour}. If the exact number of eigenvalues $e$ in the interval were known, one could simply arrange all Ritz pairs in descending order of their residuals and select the first $e$ pairs.
However, these approaches rely on a highly accurate filter and require $e$ to be chosen precisely: an overestimated value may prevent convergence, whereas an underestimated one may lead to incorrect convergence.
For this reason, these strategies are not suitable for the method proposed in this work. Instead, we propose an alternative approach.



By classical Rayleigh–Ritz theory \cite{demmel1997applied}, the error of a Ritz value is bounded by the norm of its residual.
Specifically, for each Ritz pair \((\hat\lambda,\hat v)\), there exists an exact eigenvalue \(\lambda\) of \(A\) such that
\[
 |\hat\lambda-\lambda|\le \|r\|,\qquad r:=A\hat v-\hat\lambda\hat v.
\]
Consequently, if a Ritz value \(\hat\lambda\in[a,b]\) corresponds to an exact eigenvalue \(\lambda\notin[a,b]\), then
\[
\|r\| \ge d_{[a,b]}(\hat\lambda) := \min\{|\hat\lambda-a|,|\hat\lambda-b|\}.
\]
For all Ritz values \(\hat\lambda_i\in[a,b]\), we define the threshold
\begin{equation}\label{tau_s}
    \tau_s^{(i)} := \min_j d_{[a,b]}(\hat\lambda_j^{(i)}),
\end{equation}
computed at the $i$-th iteration. This threshold ensures that any Ritz value whose residual falls below it will converge to an eigenvalue in the interval of interest.
Once eigenvalues begin to converge, Ritz pairs whose residual norms exceed \(\tau_s^{(i)}\) are excluded from the convergence test. Applying this strategy to the previous example, the algorithm terminates correctly at the $11$-th iteration. Further details of this strategy are provided in lines \ref{ACJFEAST-check} to 19 of Algorithm \ref{alg:ACJFEAST}.


\subsubsection{Deflation}

Due to the particular structure of the filter, eigenvalues located closer to the interior of the  interval of interest typically converge faster and attain smaller residuals than those near the interval boundaries. 
Without additional treatment, the residuals of already converged eigenpairs continue to decrease until reaching machine precision. 
However, this additional cost leads to redundant computations. 
To reduce computational cost, we incorporate an implicit deflation technique \cite{stewart1981simultaneous,saad2011numerical} into the algorithmic framework. 
By locking the converged Ritz values, the costs of both the filtering step and the Rayleigh–Ritz procedure in subsequent iterations can be reduced. 
The implementation of the deflation technique is demonstrated in lines \ref{ACJFEAST-orth} and \ref{ACJFEAST-lock1}-\ref{ACJFEAST-lock3} of Algorithm \ref{alg:ACJFEAST}.

It is worth noting that an overly loose convergence threshold (e.g., an absolute residual exceeding $10^{-8}$) may negatively impact the convergence of subsequent iterations.

\subsection{Other details}

To compute polynomial coefficients in the line \ref{ACJFEAST-Compute_CJ_cof} of Algorithm \ref{alg:ACJFEAST}, we estimate the largest and smallest eigenvalues of $A$ using a few iterations of the Lanczos algorithm \cite{demmel1997applied}.

As discussed in section \ref{Preliminaries}, to ensure the convergence of our algorithm, the subspace dimension must not be smaller than the number of eigenvalues in the interested interval, satisfying $p \geq e$. We estimate the approximate number of eigenvalues of \(A\) in the interval \([a,b]\), denoted by \(\tilde e\), using a trace estimation approach, and then set \(p=\lceil \mu \cdot\tilde e\rceil\) with \(\mu>0\). The key idea of trace estimation is to approximate the eigenvalue count in \([a,b]\) by computing the trace of an approximate spectral projector \(P_k\). The approach have been widely used for eigenvalue counting in related problems \cite{di2016efficient,jia2023feast}. In this work, the filter \eqref{AdaPolySI} with \(k^{(i)} = k\) is used as the projector in the trace estimation. 
It is also worth noting that the DOS method \cite{xi2018fast} provides an alternative way to estimate the number of eigenvalues.

In line \ref{ACJFEAST-orth} of Algorithm \ref{alg:ACJFEAST}, we employ the Cholesky-QR algorithm \cite{yamazaki2015mixed} for orthogonalization. 
Revisiting the partitioning scheme from Section \ref{sec:SPMM}, each process stores only $s$ consecutive rows of $X$. The entire algorithm executes in parallel, with only the first step requiring a small amount of communication. Consequently, when the condition number of $X$ is not excessively large $(\kappa(X) < 1/\varepsilon)$, Cholesky-QR is more suitable than Householder-QR algorithms for parallel computing environments.

\section{Numerical experiments}
\label{sec:Numerical_experiments}



In this section, we present a comprehensive experimental evaluation of the proposed algorithm. All experiments are designed to demonstrate the efficiency and robustness of our method (Algorithm \ref{alg:ACJFEAST}) compared to state-of-the-art methods. 

The methods under comparison include CJ-FEAST, originally proposed for SVD computations \cite{jia2023feast} and implemented here for Hermitian interior eigenvalue problems within Algorithm \ref{alg:ACJFEAST}, with the polynomial degree selected as in the original work. We also compare against the filtered subspace iteration method from the EVSL library \cite{li2019eigenvalues}, which employs a Chebyshev expansion of the Dirac delta function and determines the polynomial degree through additional constraints.



The experiments are based on the high performance implementation version (AdaPolySI) of our proposed algorithm~\cite{Yuhui2026AdaPolySI}, which has been implemented in the high-performance parallel algorithm library YHAMG \cite{yhamg2022,yuan2025cramg,ICS-MG}.
The software environment comprised the Kylin Linux operating system, MPICH-4.1.2, and the Intel 2023 compiler with the -Ofast optimization flag enabled. Intel MKL 2023.0 is used for mathematical library support.
To ensure reproducibility, each test is repeated five times with a fixed set of random seeds, and the average results are reported.


We define the maximum residual $\|R\| := \max_i\|A\hat v_i - \hat\lambda_i \hat v_i\|$ for $\hat\lambda_i \in [a, b]$.
Unless otherwise specified, the convergence criterion is set to
\[
    \|R\| \le \tau_c \cdot\|A\|,
\]
with \(\tau_c = 10^{-10}\). Here, \(\|A\|=\max\{|\lambda_{\min}|, |\lambda_{\max}|\}\), where \(\lambda_{\min}\) and \(\lambda_{\max}\) are estimated during the setup phase.

Throughout all experiments, our algorithm employs the $m$-th Lanczos damping coefficient with $m = 0.5$.  The adaptive threshold $\tau_a = 10^{-3}$, $C = 1.4,~k = \lceil 2.5 k^{(1)} \rceil$, and $\mu = 1.8$. These parameter values are selected based on preliminary testing to balance convergence speed and numerical stability.

\subsection{Sparse Matrices from real applications}\label{real_applications}
This section presents experimental results on sparse matrices selected from the \emph{SuiteSparse Matrix Collection}~\cite{TimDavis-Matrix}.
The key properties of these matrices are summarized in Table~\ref{tab:SpMatrices}. 
The fifth and sixth columns indicate the interval of interest and the number of eigenvalues computed, respectively. 
The five matrices are taken from the DFT framework. The interval is chosen as \([0.5n_0,1.5n_0]\), where \(n_0\) corresponds to the Fermi level of each Hamiltonian \cite{fang2012filtered}.
All experiments in this subsection were conducted using a single MPI process and 28 OpenMP threads on a single CPU socket.

We evaluate three methods for computing the desired eigenpairs: {Adaptive Polynomial Filtering (AdaPolySI), {CJ-FEAST}, and the {EVSL} subspace iteration method (EVSL-PSI). We note that EVSL also includes other algorithms, such as restarted Lanczos and rational filtering methods; however, in this comparison, we focus specifically on its \emph{polynomial filtering subspace iteration} method. 
The performance results are shown in Fig.~\ref{fig:Runtimes_MVs}, where we report the relative runtime and relative SpMVs of the three methods. The bar charts in the left panel additionally display the speedups of CJ-FEAST and AdaPolySI relative to EVSL-PSI. The detailed experimental results are summarized in Table~\ref{tab:Runtimes_MVs}, where AvgDeg denotes the average polynomial degree over all iterations, and Res represents the maximum residual at convergence.


The results in Fig.~\ref{fig:Runtimes_MVs} demonstrate that AdaPolySI achieves on average $\textbf{14.5}\times$ speedup over EVSL-PSI and $\textbf{2.14}\times$ over CJ-FEAST, with peak speedups of up to $25.9\times$ and $2.39\times$, respectively. Columns 3 and 4 of Table~\ref{tab:Runtimes_MVs} indicate that AdaPolySI achieves iteration counts comparable to CJ-FEAST while maintaining AvgDeg close to EVSL-PSI, resulting in significantly fewer SpMVs, as reflected in column 5.


 \begin{table}[t]
        \centering
        \caption{Sparse matrices from the SuiteSparse matrix collection.}
        \small
        \label{tab:SpMatrices}
        \begin{tabular}{llllccr}
        \toprule
        \textbf{ID} & \textbf{Matrix} & \textbf{Rows} & \textbf{\#nnz} & \textbf{Interval} & \textbf{\#evals} & \textbf{Ratio} \\
        \midrule
        1 & \texttt{Ge87H76} & 112,985 & 7,892,195 & $[-0.64, -0.0053]$  & 212 & 0.09\% \\
        2 & \texttt{Ge99H100} & 112,985 & 8,451,395 & $[-0.65, -0.0096]$  & 250 & 0.22\% \\
        3 & \texttt{Si41Ge41H72} & 185,639 & 15,011,265 & $[-0.64, -0.0028]$ & 218 & 0.12\% \\
        4 & \texttt{Si87H76} & 240,369 & 10,661,631 & $[-0.66, -0.0300]$ & 213 & 0.09\% \\
        5 & \texttt{Ga41As41H72} & 268,096 & 18,488,476 & $[-0.64, -0.0000]$ & 201 & 0.08\% \\
        \bottomrule           
        \end{tabular}
\end{table}

\begin{table}[!ht]
    \centering
    \small
    \caption{Summary of numerical results for AdaPolySI, CJ-FEAST, and EVSL-PSI.}
    \label{tab:Runtimes_MVs}
    \begin{tabular}{ccccccc}
    \toprule
        \textbf{Matrix} & \textbf{Method} & \textbf{AvgDeg} & \textbf{Iter} & \textbf{SpMVs} & \textbf{RunTime(s)} & \textbf{Res}  \\ \midrule
        ~ & AdaPolySI & 38  & 10  & 178783  & 24.4  & 9.8E-11  \\ 
        \texttt{Ge87H76} & CJ-FEAST & 167  & 7  & 439010  & 58.4  & 9.9E-11  \\ 
        ~ & EVSL-PSI & 31  & 111  & 445298  & 269.7  & 5.9E-12  \\ \hline 
        ~ & AdaPolySI & 37  & 10  & 200719  & 28.4  & 9.8E-11  \\ 
        \texttt{Ge99H100} & CJ-FEAST & 167  & 6  & 453724  & 56.3  & 9.9E-11  \\ 
        ~ & EVSL-PSI & 31  & 107  & 497261  & 331.0  & 5.0E-12  \\ \hline 
        ~ & AdaPolySI & 55  & 10  & 226664  & 54.3  & 9.8E-11  \\ 
        \texttt{Si41Ge41H72} & CJ-FEAST & 221  & 6  & 490648  & 108.2  & 9.8E-11  \\ 
        ~ & EVSL-PSI & 38  & 105  & 513410  & 752.8  & 5.5E-12  \\ \hline 
        ~ & AdaPolySI & 48  & 10  & 241880  & 55.5  & 9.8E-11  \\ 
        \texttt{Si87H76} & CJ-FEAST & 198  & 7  & 560851  & 111.3  & 9.6E-11  \\ 
        ~ & EVSL-PSI & 35  & 109  & 543923  & 550.7  & 4.5E-12  \\ \hline 
        ~ & AdaPolySI & 292  & 9  & 1117369  & 351.4  & 9.5E-11  \\ 
        \texttt{Ga41As41H72} & CJ-FEAST & 1000  & 7  & 2513615  & 809.5  & 9.9E-11  \\ 
        ~ & EVSL-PSI & 510  & 30  & 5123342  & 9087.0  & 3.8E-12  \\ \bottomrule
    \end{tabular}
\end{table}


The performance advantage of AdaPolySI over EVSL-PSI arises from several factors. 
First, AdaPolySI adopts a different filtering strategy; as illustrated in the right panel of Fig.~\ref{fig:Runtimes_MVs}, this significantly reduces the number of required SpMVs. 
Second, improvements in the algorithmic framework include a refined convergence detection mechanism that provides a more accurate assessment of the convergence status, helping avoid unnecessary iterations. 
Third, our implementations of both AdaPolySI and CJ-FEAST are based on MaSpMM implementation, whereas EVSL relies solely on sparse matrix-vector multiplication (SpMV). 
SpMM generally offers better data locality and memory access efficiency, making it faster than SpMV in practice.
Comparing with CJ-FEAST~\cite{jia2024augmented}, AdaPolySI update the order of polynomial adaptively, and therefore
requires much fewer SpMVs.


\begin{figure}[tb]
        \centering  
             \includegraphics[width=\linewidth]{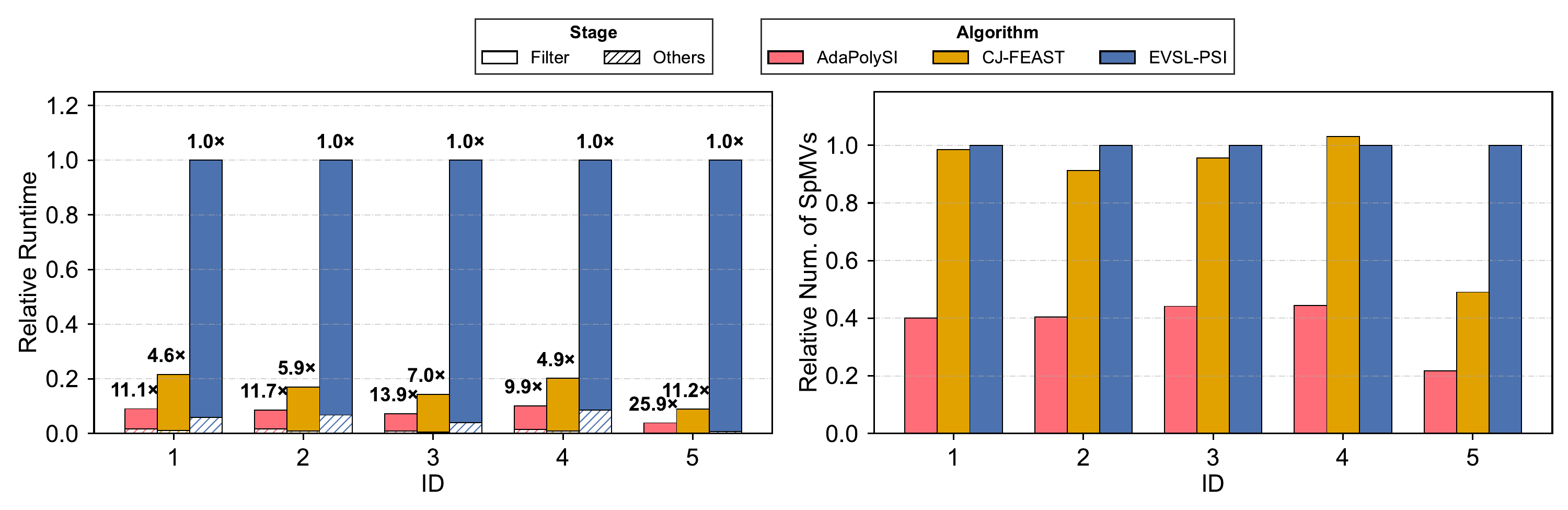} 
            \caption{The relative runtime and SpMVs of AdaPolySI and CJ-FEAST over EVSL. 
            }
            \label{fig:Runtimes_MVs}  
\end{figure}



\subsection{Parameter Sensitivity Analysis}\label{Par_Sensitivity}

The sensitivity of AdaPolySI with respect to the parameters \( m \), \( \tau_a \) and \(k\) is analyzed. We use the test cases with matrices \texttt{Ga41As41H72} and \texttt{Ge87H76} from Table \ref{tab:SpMatrices} for this evaluation. All experiments are performed using a single MPI process with 28 OpenMP threads.

To examine the effect of \( m \), we test AdaPolySI over the range \( m \in [0, 3] \) with a step size of $0.1$.
The results are shown in Fig. \ref{fig:mu_tola_sensitivity}(a). The runtime remains nearly unchanged when \( m \) is small, but gradually increases as \( m \) becomes larger. This indicates that choosing a small value of \( m \) leads to similar-and generally better—performance compared with using larger values.  

For the adaptive parameter that determines the polynomial degree, we evaluate $\tau_a$ over the range $10^{-4}$--$10^{-1}$, and the results are shown in Fig.~\ref{fig:mu_tola_sensitivity}(b).
We observe that the performance is largely insensitive to $\tau_a$, and the overall runtime remains nearly unchanged as long as $\tau_a$ is sufficiently small, e.g., less than $0.01$.
As $\tau_a$ increases, the average degree decreases, which in turn leads to a larger number of iterations.
This results in additional computational costs in the orthogonalization and Rayleigh--Ritz stages.
In particular, when $\tau_a = 0.5$, the time spent in the “Others” stage increases significantly, rendering the adaptive strategy ineffective.
Based on these observations, we set $m = 0.5$ and $\tau_a = 10^{-3}$.

To better illustrate the advantage of the adaptive strategy, we evaluate how algorithms with adaptive and fixed filters are affected by the choice of the initial polynomial degree $k$. 
We compute \(k^{(1)}\) using \eqref{compute_k_1}, set \(k=\lceil 2.5 k^{(1)} \rceil\), and define \(C_{\mathrm{r}} = \frac{C}{\alpha - \beta}\). We consider 13 values of \(C_{\mathrm{r}}\) in the range \([10, 120]\).

Fig.~\ref{fig:Degree_sensitivity} shows the runtime for different values of \(C_{\mathrm{r}}\), together with the average degree per iteration. Comparing Fig.~\ref{fig:Degree_sensitivity}(a) and (b), we observe that adaptive filtering is much less sensitive to the choice of \(k\) than fixed-degree filtering. Taking \texttt{Ge87H76} as an example, the variance of the runtime is 72.76 for the fixed-degree approach, whereas it is only 5.61 for the adaptive strategy. This indicates that the initial degree selection has a relatively minor impact on the performance of adaptive filtering. Moreover, the average degrees produced by the adaptive strategy remain very close across different settings, suggesting that adaptive filtering improves algorithmic robustness compared with fixed-degree filtering.


\begin{figure}[tb]
    \centering
    \includegraphics[width=\linewidth]{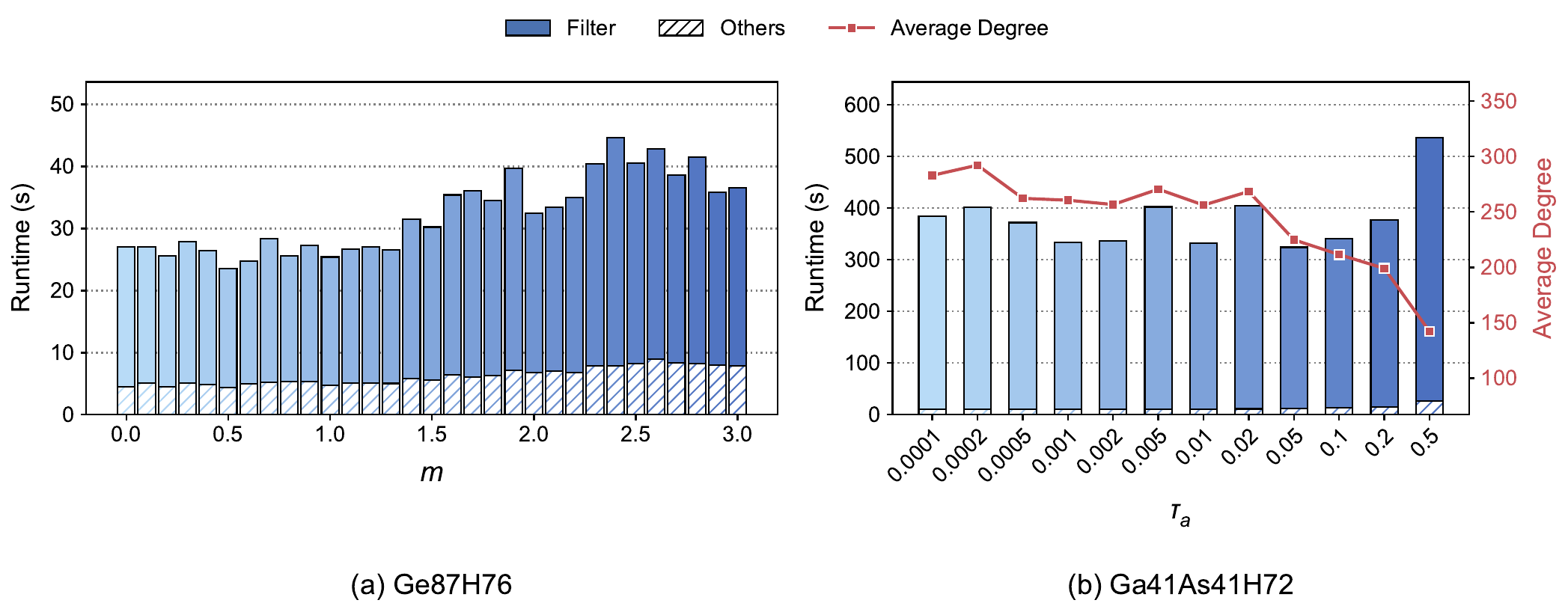}
    \caption{Runtime and average degree of AdaPolySI for different values of $m$ and $\tau_a$.}
    \label{fig:mu_tola_sensitivity}
\end{figure}

\begin{figure}[tb]
    \centering
    \includegraphics[width=\linewidth]{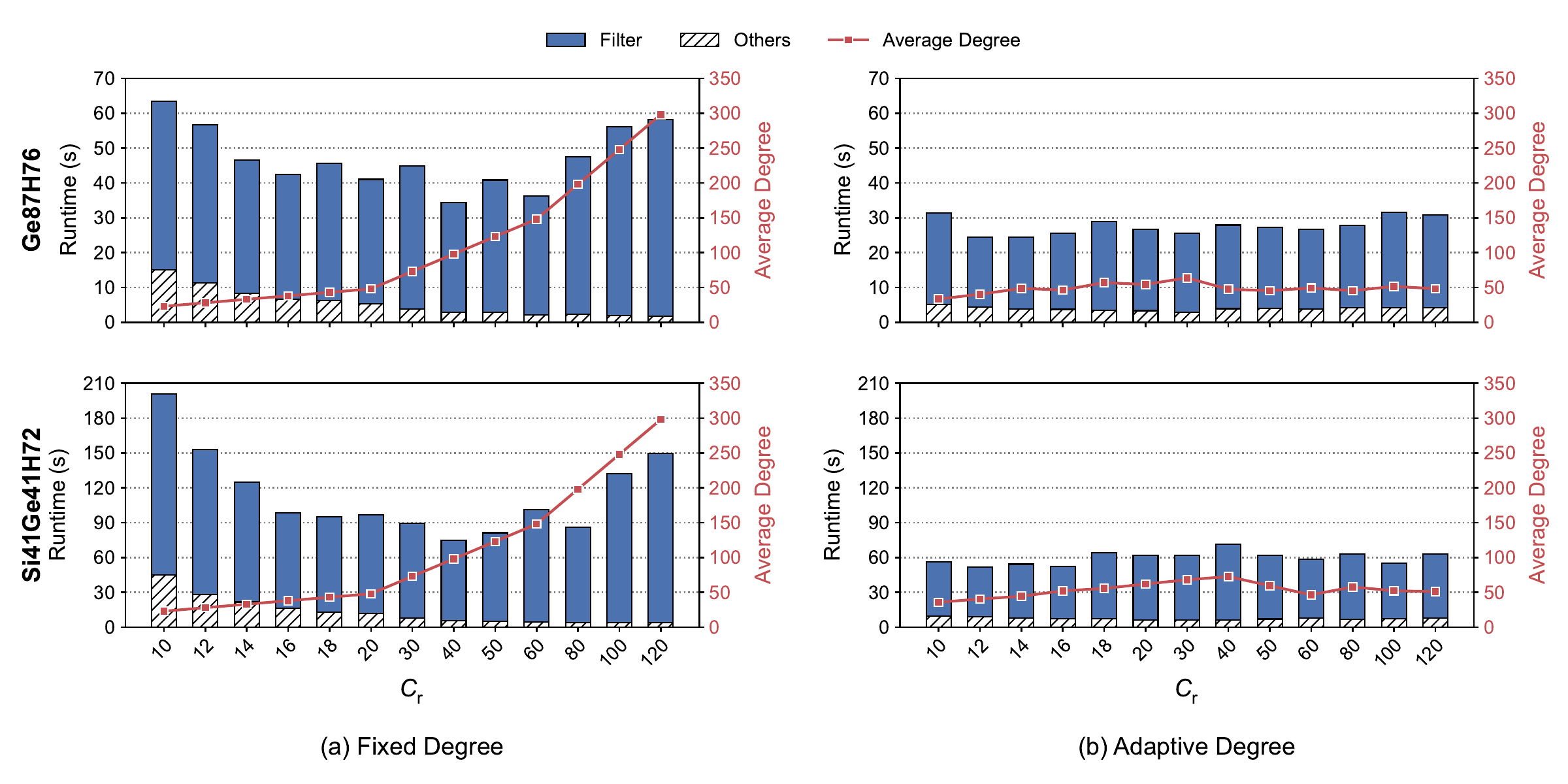}
    \caption{Runtime and average polynomial degree for varying \(C_{\mathrm{r}}\). Bars show the runtime components, and the dashed curve denotes the average degree per iteration.}
    \label{fig:Degree_sensitivity}
\end{figure}

\section{Conclusions} \label{Conclusions}

We propose an adaptive polynomial filter \eqref{AdaPolySI}, constructed via the Chebyshev expansion of the step function \eqref{step_func}, and integrate it into a subspace iteration framework to solve the Hermitian interior eigenvalue problem \eqref{ori_problem}.
We establish pointwise convergence bounds for the filtering function in both the undamped and damped settings \eqref{eq:damped_fourier_error_bound}. These results show that the error depends on $1/k$ and $J(\theta)$, and is generally inversely related to the width of the target interval.
In addition, we incorporate implicit deflation and spurious eigenvalue detection into the filtered subspace iteration framework: the former reduces redundant computations associated with converged eigenpairs, while the latter ensures correct convergence in the presence of spurious Ritz values. At the implementation level, we employ MaSpMM~\cite{MaSpMM}, a memory-aware sparse matrix--dense matrix multiplication microkernel, to accelerate the polynomial filtering step.


Numerical experiments provide a comprehensive evaluation of AdaPolySI. Section~\ref{real_applications} shows that, in a shared-memory setting, AdaPolySI achieves speedups of up to $2.39\times$ over CJ-FEAST and $25.9\times$ over the PSI solver in EVSL. Section~\ref{Par_Sensitivity} further presents a systematic sensitivity analysis with respect to $m$, $\tau_a$, and $k$, demonstrating robustness to parameter variations within their prescribed ranges.



In future work, we aim to extend the proposed adaptive filtering strategy to other eigenvalue frameworks, such as filtered Lanczos methods.

\bibliographystyle{siamplain}
\bibliography{references}

\end{document}